\documentclass[a4paper]{amsart}
\usepackage[english]{babel}
\usepackage{amsmath,amssymb,amsthm}

\usepackage{graphicx}

\usepackage[pdftex]{color}
\usepackage[bookmarks=true,hyperindex,pdftex,colorlinks,citecolor=blue]{hyperref}
\hypersetup{pdftitle={The Bishop-Phelps-Bollob\'{a}s theorem for L(L_p,L_q)},
pdfauthor={Yun Sung Choi, Sun Kwang Kim, Han Ju Lee, and Miguel Martin},pdfpagemode=UseOutlines}

\theoremstyle{plain}
\newtheorem{theorem}{Theorem}[section]
\newtheorem{corollary}[theorem]{Corollary}

\newtheorem{proposition}[theorem]{Proposition}
\newtheorem{lemma}[theorem]{Lemma}

\theoremstyle{definition}

\newtheorem{definition}[theorem]{Definition}

 \DeclareMathOperator{\re}{Re\,}
 \DeclareMathOperator{\im}{Im\,}
 \DeclareMathOperator{\e}{e}
 \DeclareMathOperator{\Id}{\mathrm{Id}}
 \DeclareMathOperator{\dist}{dist\,}
 \DeclareMathOperator{\sign}{sign}

\newcommand{\K}{\mathbb{K}}

\newcommand{\C}{\mathbb{C}}
\newcommand{\R}{\mathbb{R}}

\newcommand{\inner}[1]{\ensuremath{\left\langle #1\right\rangle}}

\newcommand{\norm}[1]{\ensuremath{\lVert#1\rVert}}

\newcommand{\nor}[1]{\ensuremath{\left\|#1\right\|}}

\newcommand{\eps}{\varepsilon}

\renewcommand{\leq}{\leqslant}
\renewcommand{\geq}{\geqslant}

\begin{document}

\dedicatory{Dedicated to the memory of Robert R.\ Phelps}

\title[The Bishop-Phelps-Bollob\'{a}s theorem for operators on $\L_1(\mu)$]{The Bishop-Phelps-Bollob\'{a}s theorem for operators on $\boldsymbol{L_1(\mu)}$}

\author[Choi]{Yun Sung Choi}
\address[Choi]{Department of Mathematics, POSTECH, Pohang (790-784), Republic of Korea}
\email{\texttt{mathchoi@postech.ac.kr}}

\author[Kim]{Sun Kwang Kim}
\address[Kim]{School of Mathematics, Korea Institute for Advanced Study (KIAS), 85 Hoegiro, Dongdaemun-gu, Seoul 130-722, Republic of Korea}
\email{\texttt{lineksk@gmail.com}}

\author[Lee]{Han Ju Lee}
\address[Lee]{Department of Mathematics Education,
Dongguk University - Seoul, 100-715 Seoul, Republic of Korea}
\email{\texttt{hanjulee@dongguk.edu}}

\author[Mart\'{\i}n]{Miguel Mart\'{\i}n}
\address[Mart\'{\i}n]{Departamento de An\'{a}lisis Matem\'{a}tico,
Facultad de Ciencias,
Universidad de Granada,
E-18071 Granada, Spain}
\email{\texttt{mmartins@ugr.es}}

\subjclass[2010]{Primary 46B20; Secondary 46B04, 46B22}
\keywords{Banach space, approximation, norm-attaining operators, Bishop-Phelps-Bollob\'{a}s theorem.}
\thanks{First named author partially supported by Basic Science Research Program through the National Research Foundation of Korea (NRF) funded by the Ministry of Education, Science and Technology (No. 2010-0008543), and also by Priority Research Centers Program through the National Research Foundation of Korea (NRF) funded by the Ministry of Education, Science and Technology (MEST) (No. 2012047640).
Third named author partially supported by Basic Science Research Program through the National Research Foundation of Korea(NRF) funded by the Ministry of Education, Science and Technology (2012R1A1A1006869).
Fourth named author partially supported by Spanish MINECO and FEDER project no.\ MTM2012-31755, Junta de Andaluc\'{\i}a and FEDER grants FQM-185 and P09-FQM-4911, and by CEI-Granada 2009.}

\begin{abstract}
In this paper we show that the Bishop-Phelps-Bollob\'as theorem holds for $\mathcal{L}(L_1(\mu), L_1(\nu))$ for all measures $\mu$ and $\nu$ and also holds for $\mathcal{L}(L_1(\mu),L_\infty(\nu))$ for every arbitrary measure $\mu$ and every localizable measure $\nu$. Finally, we show that the Bishop-Phelps-Bollob\'as theorem holds for two classes of bounded linear operators from a real $L_1(\mu)$ into a real $C(K)$ if $\mu$ is a finite measure and $K$ is a compact Hausdorff space. In particular, one of the classes includes all Bochner representable operators and all  weakly compact operators.
\end{abstract}

\date{March 25th, 2013}

\maketitle

\section{Introduction}

The celebrated Bishop-Phelps theorem of 1961 \cite{BP} states that for a Banach space $X$, every element in its dual space $X^*$ can be approximated by ones that attain their norms. Since then, there has been an extensive research to extend this result to bounded linear operators between Banach spaces  \cite{Bou,JW,Lin, Par, Sch} and non-linear mappings \cite{AAP, AGM, AFW,
 C, CK1, KL}. On the other hand, Bollob\'as \cite{Bol}, motivated by problems arising in the theory of numerical ranges, sharpened the Bishop-Phelps theorem in 1970, and got what is nowadays called the Bishop-Phelps-Bollob\'as theorem. Previous to presenting this result, let us introduce some notations. Given a (real or complex) Banach space $X$, we write $B_X$ for the unit ball, $S_X$ for its unit sphere, and $X^*$ for the topological dual space of $X$. If $Y$ is another Banach space, we write $\mathcal{L}(X,Y)$ to denote the space of all bounded linear operators from $X$ into $Y$.

\begin{theorem}[Bishop-Phelps-Bollob\'{a}s theorem]\label{thm:BPBTheorem}
Let $X$ be a Banach space. If $x\in S_X$ and $x^*\in S_{X^*}$ satisfy $|x^*(x)-1|< \eps^2/4$, then there exist $y\in
S_X$ and $y^*\in S_{X^*}$ such that $y^*(y)=1$, $\|x^*-y^*\|<\eps$ and $\|x-y\|<\eps$.
\end{theorem}

In 2008, Acosta, Aron, Garc\'ia and Maestre \cite{AAGM2} introduced the Bishop-Phelps-Bollob\'as property to study
extensions of the theorem above to operators between Banach spaces.

\begin{definition}
Let $X$ and $Y$ be Banach spaces. The pair $(X,Y)$ is said to have the \emph{Bishop-Phelps-Bollob\'as property} (\emph{BPBp}) if for every $0<\eps<1$, there is $\eta(\eps)>0$ such that for every $T\in \mathcal{L}(X,Y)$ with $\|T\|=1$ and $x_0 \in S_X$ satisfying $\|T(x_0)\|>1-\eta(\eps)$, there exist $y_0\in S_X$  and $S\in \mathcal{L}(X,Y)$ with $\|S\|=1$ satisfying the following conditions:
\[
\|Sy_0\|=1, \qquad \|y_0 - x_0 \|<\eps,  \quad \text{and} \quad \|S- T\|<\eps.
\]
In this case, we also say that the Bishop-Phelps-Bollob\'{a}s theorem holds for $\mathcal{L}(X,Y)$.
\end{definition}

This property has been studied by many authors. See for instance \cite{ABGM, ACK, ACKLM, CasGuiKad, CK2, ChoiKimSK, Kim-c_0, KimLee}. Observe that the BPBp of a pair $(X,Y)$ implies  obviously that the set of norm attaining operators is dense in $\mathcal{L}(X,Y)$. However, its converse is false, as shown by the pair $(X,Y)$ where $X$ is the $2$-dimensional $L_1$-space and $Y$ is a strictly, but not uniformly convex space (see \cite{AAGM2} or \cite{ACKLM}). Let us also comment that the Bishop-Phelps-Bollob\'{a}s theorem states that the pair $(X,\K)$ has the Bishop-Phelps-Bollob\'{a}s property for every Banach space $X$ ($\K$ is the base scalar field $\R$ or $\C$).

In this paper we first deal with the problem of when the pair $(L_p(\mu),L_q(\nu))$ has the BPBp. Let us start with a presentation of both already known results and our new results. Iwanik \cite{Iwa} showed in 1979 that the set of norm-attaining operators from $L_1(\mu)$ to $L_1(\nu)$ is dense in the space $\mathcal{L}(L_1(\mu),L_1(\nu))$ for arbitrary measures $\mu$ and $\nu$. Our first main result in this paper is that the pair $\mathcal{L}(L_1(\mu),L_1(\nu))$ has the BPBp. This is the content of section~\ref{sec:L_1-L_1}.

On the other hand, Aron et al.\ \cite{ACGM} showed that if $\mu$ is a $\sigma$-finite measure, then the pair $(L_1(\mu), L_\infty[0,1])$ has the BPBp, improving a result of Finet and Pay\'{a} \cite{FP} about the denseness of norm-attaining operators. We generalize this result in section~\ref{sec:L_1-L_infty} showing that $(L_1(\mu), L_\infty(\nu))$ has the BPBp for every measure $\mu$ and every localizable measure $\nu$. This is also a strengthening of a result of Pay\'a and Saleh \cite{PS} which stated only the denseness of norm-attaining operators.

One of the tools used to prove the results above is the fact that one can reduce the proofs to some particular measures. We develop this idea in section~\ref{sec:preliminary}, where, as its first easy application, we extend to arbitrary measures $\mu$ the result in \cite{CK2} that $(L_1(\mu),L_p(\nu))$ has the BPBp for $\sigma$-finite measures $\mu$.

The following result summarizes all what is known about the BPBp for the pair $(L_p(\mu),L_q(\nu))$.

\begin{corollary}
The pair $(L_p(\mu), L_q(\nu))$ has the BPBp
\begin{itemize}
\item[(1)] for all measures $\mu$ and $\nu$ if $p=1$ and $1\leq q<\infty$.
\item[(2)] for any measure $\mu$ and any localizable measure $\nu$ if $p=1$, $q=\infty$.
\item[(3)] for all measures $\mu$ and $\nu$ if $1<p<\infty$ and $1\leq q\leq \infty$.
\item[(4)] for all measures $\mu$ and $\nu$ if $p=\infty$, $q=\infty$, in the real case.
\end{itemize}
\end{corollary}

(1) and (2) follows from the results of this paper (Corollary~\ref{cor:L1Lq}, Theorem~\ref{thm:main0} and Theorem~\ref{thm:localizable}). Since $L_p(\mu)$ is uniformly convex when $1<p<\infty$, (3) follows from \cite{ABGM, KimLee} in the $\sigma$-finite case, generalized here to arbitrary measures $\mu$ (Corollary~\ref{cor:L1Lq}). Finally, (4) follows from \cite{ABCCKLLM}, because every $L_\infty$ space is isometrically isomorphic to a $C(K)$ space.

As far as we know, the cases $(L_\infty(\mu), L_q(\nu))$ for $1\leq q<\infty$ and the complex case of (4) remain open.

Let $\mu$ be a finite measure. Since any $L_\infty$ space is isometrically isomorphic to $C(K)$ for some compact Hausdorff space $K$, it is natural to ask when $(L_1(\mu), C(K))$ has the BPBp. Schachermayer \cite{Sch2} showed that the set of all norm-attaining operators is not dense in $\mathcal{L}(L_1[0,1],C[0,1])$. Hence,  $(L_1[0,1], C[0,1])$ cannot have the BPBp. On the other hand, Johnson and Wolfe \cite{JW} proved that if $X$ is a Banach space and if either $Y$ or $Y^*$ is a $L_1(\mu)$ space, then every compact operator from $X$ into $Y$ can be approximated by norm-attaining finite-rank operators. They also showed that every weakly compact operator from $L_1(\mu)$ into $C(K)$ can be approximated by norm-attaining weakly compact ones. In this direction, Acosta et al.\ have shown that $(L_1(\mu), Y)$ has the BPBp for representable operators (in particular, for weakly compact operators) if $(\ell_1, Y)$ has the BPBp, and this is the case of $Y=C(K)$ \cite{ABGKM}.

On the other hand, Iwanik \cite{Iwa2} studied two classes of bounded linear operators from a real $L_1(\mu)$ space to a real $C(K)$ space such that every element of each class can be approximated by norm-attaining elements, and showed that one of the classes strictly contains all Bochner representable operators and all weakly compact operators. In section~\ref{sec:L_1-C(K)}, we deal with  Bishop-Phelps-Bollob\'as versions of these Iwanik's results. In particular, we show that for every $0<\eps<1$, there is $\eta(\eps)>0$ such that if $T\in\mathcal{L}(L_1(\mu),C(K))$ with $\|T\|=1$ is Bochner representable (resp.\ weakly compact) and $f_0\in S_{L_1(\mu)}$ satisfy $\|T f_0\|>1-\eta(\eps)$, then there is a Bochner representable (resp.\ weakly compact) operator $S\in\mathcal{L}(L_1(\mu),C(K))$ and $f\in S_{L_1(\mu)}$ such that $\|S f\|=\|S\|=1$, $\|S-T\|<\eps$ and $\|f-f_0\|<\eps$.

Let us finally comment that the proofs presented in sections \ref{sec:L_1-L_1} and \ref{sec:L_1-L_infty} are written for the complex case. Their corresponding proofs for the real case are easily obtained, even easier, from the ones presented there.

\section{Some preliminary results}\label{sec:preliminary}
We start with some terminologies and known facts about $L_1(\mu)$. Suppose that $(\Omega,\Sigma,\mu)$ is an arbitrary measure space and put $X=L_1(\mu)$. Suppose $G$ is a countable subset of $X$. Since the closed linear span $[G]$ of $G$ is separable, we may assume that $[G]$ is the closed linear span of a countable set $\{ \chi_{E_n} \}$ of characteristic functions of measurable subsets with finite positive measure. Let $E= \bigcup_{n} E_n$ and $Z = \{ f\chi_E\, :\, f\in X\}$. Then, $Z= L_1(\mu|_E)$, where $\mu|_E$ is restriction of the measure $\mu$ to the $\sigma$-algebra $\Sigma|_E = \{ E\cap A\, :\, A\in \Sigma\}$. Since $\mu|_E$ is $\sigma$-finite, $Z$ is isometrically (lattice) isomorphic to $L_1(m)$ for some positive finite Borel regular measure $m$ defined on a compact Hausdorff space by the Kakutani representation theorem (see \cite{Lac} for a reference). This space $Z$ is called the {\it band} generated by $G$, and the {\it canonical band projection} $P:X\longrightarrow Z$, defined by $P(f) := f\chi_E$ for $f\in X$, satisfies $\|f\| = \|Pf\| + \|(\Id-P)f\|$ for all $f\in X$. For more details, we refer the reader to the classical books \cite{Lac, Schaefer}.

Next, we present the following equivalent formulation of the BPBp from \cite{ACKLM} which helps to better understand the property and will be useful for our preliminary results. Given a pair $(X,Y)$ of Banach spaces, let
$$
\Pi(X,Y)= \{(x,T)\in X\times \mathcal{L}(X,Y)\,:\, \|T\|=\|x\|=\|Tx\|=1\}
$$
and define, for  $0<\eps<1$,
\begin{equation*}
\eta(X,Y)(\eps)=\inf\bigl\{1-\|Tx\|\,:\ x\in S_X,\, T\in \mathcal{L}(X,Y),\, \|T\|=1,\  \dist\bigl((x,T),\Pi(X,Y)\bigr)\geq\eps\bigr\},
\end{equation*}
where $\dist\bigl((x,T),\Pi(X,Y)\bigr)= \inf\bigl\{\max\{\|x-y\|,\|T-S\|\}\ :\ (y,S)\in \Pi(X,Y)\bigr\}$.
Equivalently, for every $\eps\in (0,1)$, $\eta(X,Y)(\eps)$ is the supremum of those $\xi\geq 0$ such that whenever $T\in \mathcal{L}(X,Y)$ with $\|T\|=1$ and $x\in S_X$ satisfy $\|Tx\|\geq 1-\xi$, then there exists $(y,S)\in \Pi(X,Y)$ with $\|T-S\|\leq \eps$ and $\|x-y\|\leq \eps$. It is clear that $(X,Y)$ has the BPBp if and only if $\eta(X,Y)(\eps)>0$ for all $0<\eps<1$.

Our first preliminary result deals with operators acting on an $L_1(\mu)$ space and shows that the proof of some results can be reduced to the case when $\mu$ is a positive finite Borel regular measure defined on a compact Hausdorff space.

\begin{proposition}\label{prop:reduction}
Let $Y$ be a Banach space. Suppose that there is a function $\eta: (0,1)\longrightarrow (0, \infty)$ such that
$$
\eta\bigl(L_1(m),Y\bigr)(\eps)\geq \eta(\eps)>0 \qquad (0<\eps<1)
$$
for every positive finite Borel regular measure $m$ defined on a compact Hausdorff space. Then, for every measure $\mu$, the pair $(L_1(\mu),Y)$ has the BPBp with $\eta\bigl(L_1(\mu), Y\bigr)\geq \eta$.

Moreover, if $Y=L_1(\nu)$ for an arbitrary measure $\nu$, then it is enough to show that
$$
\eta\bigl(L_1(m_1),L_1(m_2)\bigr)(\eps)\geq \eta(\eps)>0 \qquad (0<\eps<1)
$$
for all positive finite Borel regular measures $m_1$ and $m_2$ defined on Hausdorff compact spaces in order to get that $(L_1(\mu),L_1(\nu))$ has the BPBp with $\eta\bigl(L_1(\mu), L_1(\nu)\bigr)\geq \eta$.
\end{proposition}

\begin{proof}
Let $0<\eps<1$. Suppose that $T\in\mathcal{L}(L_1(\mu),Y)$ is a norm-one operator and $f_0\in S_X$ satisfy that $\|Tf_0\|>1-\eta(\eps)$. Let $\{f_n)_{n=1}^\infty$ be a sequence in $X$ such that $\|f_n\|\leq 1$ for all $n$ and $\lim_{n\to \infty} \|Tf_n\|=\|T\|=1$. The band $X_1$ generated  by $\{f_n\,:\, n\geq 0\}$ is isometric to $L_1(J, m)$ for a finite positive Borel regular measure $m$ defined on a compact Hausdorff space $J$ by the Kakutani representation theorem. Let $T_1$ be the restriction of $T$ to $X_1$. Then $\|T_1\|=1$ and $\|T_1 f_0\| > 1- \eta(\eps)$. By the assumption, there exist a norm-one operator $S_1: X_1 \longrightarrow Y$ and $g\in S_{X_1}$ such that $\|S_1g\|=1$, $\|T_1-S_1\|< \eps$ and $\|f-g\|< \eps$. Let $P$ denote the canonical band projection from $L_1(\mu)$ onto $X_1$. Then $S := S_1P + T(\Id-P)$ is a norm-one operator from $L_1(\mu)$ to $Y$, $g$ can be viewed as a norm-one element in $S_{L_1(\mu)}$ (just extending by $0$), $\|Sg\|=1$, $\|S- T\|< \eps$ and $\|f-g\|< \eps$. This completes the proof of the first part of the proposition.

In the case when $Y=L_1(\nu)$, we observe that the image $T(X_1)$ is also contained in a band $Y_1$ of $L_1(\nu)$ which, again, is isometric to $L_1(m_2)$ for a finite positive Borel regular measure $m_2$ on a compact Hausdorff space $J_2$. Now, we work with the restriction of $T$ to $X_1$ with values in $Y_1$, we follow the proof of the first part and finally we consider the operators $S$ there as an operator with values in $L_1(\nu)$ (just composing with the formal inclusion of $Y_1$ into $L_1(\nu)$).
\end{proof}

Since for every positive finite Borel regular measure $m$ defined on a compact Hausdorff space, $L_1(m)$ is isometric to $L_1(\mu)$ for a probability measure $\mu$, we get the following.

\begin{corollary}\label{cor:main1} Let $Y$ be a Banach space.
Suppose that there is a strictly positive function $\eta: (0,1)\longrightarrow (0, \infty)$ such that $\eta\bigl(L_1(\mu_1),Y\bigr) \geq \eta$ for every probability measure $\mu_1$. Then $(L_1(\mu),Y)$ has the BPBp for every measure $\mu$, with $\eta\bigl(L_1(\mu),Y\bigr)\geq \eta$.
\end{corollary}

Let us present the first application of the above results. For a $\sigma$-finite measure $\mu_1$, it is shown in \cite{CK2} that $(L_1(\mu_1), Y)$ has the BPBp if $Y$ has the Radon-Nikod\'ym property and $(\ell_1,Y)$ has the BPBp. By following the proof of  \cite[Theorem~2.2]{CK2}, we conclude that there is a strictly positive function $\eta_Y: (0,1)\longrightarrow (0,\infty)$ such that $\eta(L_1(\mu_1), Y)\geq  \eta_Y$ for every probability measure $\mu_1$. Therefore, the corollary above provides the same result without the assumption of $\sigma$-finiteness. We also recall that $L_q(\nu)$ is uniformly convex for all $1<q<\infty$ and for all measures $\nu$, so it has the Radon-Nikod\'ym property and $(\ell_1,L_q(\nu))$ has the BPBp \cite{AAGM2}. Hence we get the following.

\begin{corollary}\label{cor:L1Lq} Let $\mu$ be an arbitrary measure. If $Y$ is a Banach space with the Radon-Nikod\'ym property and such that $(\ell_1,Y)$ has the BPBp, then the pair $(L_1(\mu), Y)$ has the BPBp. In particular, $(L_1(\mu), L_q(\nu))$ has the BPBp for all $1<q<\infty$ and all arbitrary measures $\nu$.
\end{corollary}

We now deal with operators with values on an $\ell_\infty$-sum of Banach spaces, presenting the following result from \cite{ACKLM} which we will use in section~\ref{sec:L_1-L_infty}. Given a family $\{Y_j\, :\, j\in J\}$ of Banach spaces, we
denote by $\left[\bigoplus_{j\in J} Y_j\right]_{\ell_\infty}$ the
$\ell_\infty$-sum of the family.

\begin{proposition}[\mbox{\cite{ACKLM}}]\label{prop:ACKLM}
Let $X$ be a Banach space and let $\{Y_j\, :\, j\in J\}$ be a family of Banach spaces and let $Y=\left[\bigoplus_{j\in J} Y_j\right]_{\ell_\infty}$ denote their $\ell_\infty$-sum. If $\inf\limits_{j\in J} \eta(X,Y_j)(\eps)>0$ for all $0<\eps<1$, then
$\left(X, Y\right)$ has the BPBp with
$$
\eta(X,Y)= \inf\limits_{j\in J} \eta(X,Y_j).
$$
\end{proposition}

We will use this result for operators with values in $L_\infty(\nu)$. To present the result, we fist recall that given a localizable measure $\nu$, we have the following representation
\begin{equation}\label{eq:repre_L_infty}
L_\infty(\nu) \equiv \left[ \bigoplus\nolimits_{j\in J} Y_j \right]_{\ell_\infty},
\end{equation}
where each space $Y_j$ is either $1$-dimensional or of the form $L_\infty([0,1]^\Lambda)$ for some finite or infinite set $\Lambda$ and $[0,1]^\Lambda$ is endowed with the product measure of the Lebesgue measures. We refer to \cite{Lac} for its background. With this in mind, the following corollary follows from the proposition above.

\begin{corollary}\label{cor:to_L_infty-localizable}
Let $X$ be a Banach space. Suppose that there is a strictly positive function $\eta:(0,1)\longrightarrow (0,\infty)$ such that
$$
\eta\bigl(X,L_\infty([0,1]^\Lambda) \bigr)(\eps)\geq \eta(\eps) \qquad (0<\eps<1)
$$
for every finite or infinite set $\Lambda$. Then the pair $(X,L_\infty(\nu))$ has the BPBp for every localizable measure $\nu$ with
$$
\eta\bigl(X,L_\infty(\nu)\bigr)(\eps)\geq \min\{\eta(\eps),\eps^2/2\} \qquad (0<\eps<1).
$$
\end{corollary}

The proof is just an application of Proposition~\ref{prop:ACKLM}, the representation formula given in \eqref{eq:repre_L_infty} and the Bishop-Phelps-Bollob\'{a}s theorem (Theorem~\ref{thm:BPBTheorem}).

Let us comment that the analogue of Proposition~\ref{prop:ACKLM} is false for $\ell_1$-sums in the domain space (see \cite{ACKLM}), so Proposition~\ref{prop:reduction} cannot be derived directly from the decomposition of $L_1(\mu)$ spaces analogous to \eqref{eq:repre_L_infty}.

Before finishing this section, we state the following lemma of \cite{AAGM2} which we will frequently use afterwards.

\begin{lemma}[\mbox{\cite[Lemma~3.3]{AAGM2}}]\label{elementary}
Let $\{c_n\}$ be a sequence of complex numbers with $|c_n|\leq 1$ for every $n$, and let $\eta>0$ be such that for a convex series $\sum \alpha_n$, $\re \sum_{n=1}^\infty \alpha_n c_n >1-\eta$. Then for every $0<r<1$, the set $A : = \{ i \in \mathbb{N} : \re c_i > r \}$, satisfies the estimate
\[ \sum_{i\in A} \alpha_i \geq 1-\frac{\eta}{1-r}.\]
\end{lemma}

\section{The Bishop-Phelps-Bollob\'as property of $(L_1(\mu),L_1(\nu))$}\label{sec:L_1-L_1}

Our goal in this section is to prove the following result.

\begin{theorem}\label{thm:main0}
Let $\mu$ and $\nu$ be arbitrary measures. Then the pair $(L_1(\mu), L_1(\nu))$ has the BPBp. Moreover, there exists a strictly positive function $\eta:(0,1)\longrightarrow (0,\infty)$ such that
$$
\eta\bigl(L_1(\mu),L_1(\nu)\bigr)(\eps)\geq \eta(\eps) \qquad (0<\eps<1).
$$
\end{theorem}

By Proposition~\ref{prop:reduction}, it is enough to get the result for finite regular positive Borel measures defined on compact Hausdorff spaces. Therefore, Theorem~\ref{thm:main0} follows directly from the next result.

\begin{theorem}\label{thm:main1}
Let $m_1$ and $m_2$ be finite regular positive Borel measures on compact Hausdorff spaces $J_1$ and $J_2$, respectively.
Let $0<\eps<1$ and suppose that $T\in \mathcal{L}(L_1(m_1), L_1(m_2))$ with $\|T\|=1$ and $f_0\in S_{L_1(m_1)}$ satisfy $\|Tf_0\|>1-\frac{\eps^{18}}{5^32^{27}}$. Then there are $S\in S_{L(L_1(m_1), L_1(m_2))}$ and $g\in S_{L_1(m_1)}$ such that
\[
\|Sg\|=1,\quad \|f-g\|< 4\eps\quad \text{and} \quad \|T-S\|< 4\sqrt\eps.
\]
\end{theorem}

Prior to presenting the proof of this theorem, we have to recall the following representation result for operators from $L_1(m_1)$ into $L_1(m_2)$. As we announced in the introduction, we deal with only complex spaces, being the real case easily deductible, even easier, from the proof of the complex case.

Let $m_1$ and $m_2$ be finite regular positive Borel measures on compact Hausdorff spaces $J_1$ and $J_2$, respectively. For a complex-valued Borel measure $\mu$ on the product space $J_1\times J_2$, we define their marginal measures $\mu^i$ on $J_i$ ($i=1,2$) as follows:
\[
\mu^1(A) = \mu(A\times J_2) \quad \text{  and  } \quad \mu^2(B)= \mu(J_1\times B),
\]
where $A$ and $B$ are Borel measurable subsets of $J_1$ and $J_2$, respectively.

Let $M(m_1, m_2)$ be the complex Banach lattice consisting of all complex-valued Borel measures $\mu$ on the product space $J_1\times J_2$ such that each $|\mu|^i$ is absolutely continuous with respect to $m_i$ for $i=1,2$ with the norm  \[\nor{ \frac{d|\mu|^1}{dm_1}}_\infty.\] It is clear that to each $\mu\in M(m_1, m_2)$ there corresponds a unique bounded linear operator $T_\mu\in L(L_1(m_1), L_1(m_2))$ defined by
 \[ \inner{T_\mu (f), g}  = \int_{J_1\times J_2} f(x)g(y) \, d\mu(x,y),\]
 where $f\in L_1(m_1)$ and $g\in L_\infty(m_2)$.
Iwanik \cite{Iwa} showed that the mapping $\mu\longmapsto T_\mu$ is a surjective lattice isomorphism and
\[ \|T_\mu\| = \nor{ \frac{d|\mu|^1}{dm_1}}_\infty.\]
Even though he showed this for the real case, it can be easily generalized to the complex case. For details, see \cite[Theorem~1]{Iwa} and  \cite[IV Theorem 1.5 (ii), Corollary 2]{Schaefer}.

Since the proof of Theorem~\ref{thm:main1} is complicated, we divide it into the following two lemmas.

\begin{lemma}\label{lem:complex1}
Let $0<\eps<1$. Suppose that $T_\mu$ is an element of $\mathcal{L}(L_1(m_1), L_1(m_2))$ with $\|T_\mu\|=1$ for some $\mu\in M(m_1, m_2)$ and that $f_0\in S_{L_1(m_1)}$ is a nonnegative simple function such that $\|T_\mu f_0 \|> 1-\frac{\eps^3}{2^6}$. Then there are a norm-one bounded linear operator $T_\nu$ for some $\nu\in M(m_1, m_2)$ and a nonnegative simple function $f_1$ in $S_{L_1(m_1)}$ such that
\[ \norm{T_\mu - T_\nu}<\eps,\ \ \  \|f_1- f_0\|< 3\eps\]
and we have, for all $x\in {\rm supp}( f_1)$,
\[\frac{d|\nu|^1}{dm_1}(x)=1.\]
\end{lemma}

\begin{proof}
Let $f_0 = \sum_{j=1}^n \alpha_j \frac{\chi_{B_j}}{m_1(B_j)}$, where $\{B_j\}_{j=1}^n$ are mutually disjoint Borel subsets of $J_1$, $\alpha_j\geq 0$ and  $m_1(B_j)>0$ for all $1\leq j\leq n$, and $\sum_{j=1}^n \alpha_j=1$.  Let $D=\{ x\in J_1 : \frac{d|\mu|^1}{dm_1}(x) > 1-\frac{\eps}8\}$. It is clear that $m_1(D)>0$. Since $\|T_\mu f_0 \|> 1-\frac{\eps^3}{2^6}$, there is $g_0\in S_{L_\infty(m_2)}$ such that
\[\re\inner{T_\mu f_0, g_0} > 1-\frac{\eps^3}{2^6}.\]
Let
\[ J = \left\{ j\in\{1,\ldots,n\}\, :\, \frac{1}{m_1(B_j)} \int_{B_j} \frac{d|\mu|^1}{dm_1}(x) \,dm_1(x) > 1-\frac{\eps^2}{2^6}\right\}.
\]
Then we have
\[\sum_{j\in J} \alpha_j \geq 1-\eps>0.\]
Indeed, since
\begin{align*}
1-\frac{\eps^3}{2^6} &< \re\inner{T_\mu f_0, g_0} = \re \int_{J_1\times J_2} f_0(x) g_0(y)\, d\mu(x,y)\\
&\leq \int_{J_1\times J_2}  |f_0(x)| \, d|\mu|(x,y) = \int_{J_1} f_0(x)\, d|\mu|^1(x)\\
&=\sum_{j=1}^n \alpha_j \frac{1}{m_1(B_j)} \int_{B_j} d|\mu|^1(x)=\sum_{j=1}^n \alpha_j \frac{1}{m_1(B_j)} \int_{B_j} \frac{d|\mu|^1}{dm_1}(x)\, dm_1(x),
\end{align*} we have $\sum_{j\in J} \alpha_j \geq 1-\eps>0$ by Lemma~\ref{elementary}.
Note also that for each $j\in J$,
\begin{align*}
1-\frac{\eps^2}{2^6}  &< \frac{1}{m_1(B_j)} \int_{B_j} \frac{d|\mu|^1}{dm_1}(x)\, dm_1(x) \\
&= \frac{1}{m_1(B_j)} \int_{B_j\cap D} \frac{d|\mu|^1}{dm_1}(x)\, dm_1(x) + \frac{1}{m_1(B_j)} \int_{B_j\setminus D} \frac{d|\mu|^1}{dm_1}(x)\, dm_1(x) \\
&\leq \frac{m_1(B_j\cap D)}{m_1(B_j)} + \left( 1-\frac{\eps}8\right) \frac{m_1(B_j\setminus D)}{m_1(B_j)}\\
&=1-\frac{\eps}8 \frac{m_1(B_j\setminus D)}{m_1(B_j)}.
\end{align*} Hence we deduce that, for all $j\in J$,
\[\frac{m_1(B_j\setminus D)}{m_1(B_j)}\leq \frac{\eps}8.\]
Let $\tilde{B_j}=B_j\cap D$ and $\beta_j =\frac{ \alpha_j }{\sum_{j\in J} \alpha_j}$ for all $j\in J$ and define \[ f_1 = \sum_{j\in J} \beta_j \frac{\chi_{\tilde{B_j}}}{m_1(\tilde{B_j} )}.\]
It is clear that $f_1$ is a nonnegative element in $S_{L_1(m_1)}$ and
\begin{align*}
\|f_0 - f_1\| &\leq \nor{ \sum_{j\in J} \alpha_j \frac{\chi_{B_j}}{m_1(B_j)} - \sum_{j\in J} \beta_j \frac{\chi_{\tilde{B_j}}}{m_1(\tilde{B_j} )} } + \nor{ \sum_{j \in \{1, \dots, n\} \setminus J} \alpha_j \frac{\chi_{B_j}}{m_1(B_j)}  }\\
&\leq \nor{ \sum_{j\in J} \alpha_j \left( \frac{\chi_{B_j}}{m_1(B_j)} - \frac{\chi_{\tilde{B_j}}}{m_1(\tilde{B_j} )} \right) } + \nor{ \sum_{j\in J} (\alpha_j-\beta_j) \frac{\chi_{\tilde{B_j}}}{m_1(\tilde{B_j} )}  } +  \sum_{j \in \{1, \dots, n\} \setminus J} \alpha_j \\
&< \nor{ \sum_{j\in J} \alpha_j \left( \frac{\chi_{B_j}}{m_1(B_j)} - \frac{\chi_{\tilde{B_j}}}{m_1(\tilde{B_j} )} \right) } + \sum_{j\in J} |\alpha_j-\beta_j|  +\eps\\
&\leq  \nor{ \sum_{j\in J} \alpha_j \left( \frac{\chi_{B_j}}{m_1(B_j)} - \frac{\chi_{\tilde{B_j}}}{m_1(B_j)} \right) } + \nor{ \sum_{j\in J} \alpha_j \left( \frac{\chi_{\tilde{B_j}}}{m_1(B_j)} - \frac{\chi_{\tilde{B_j}}}{m_1(\tilde{B_j} )} \right) } + \left|1-\sum_{j\in J} \alpha_j\right| + \eps\\
&\leq 2 \sum_{j\in J} \alpha_j  \frac{m_1(B_j\setminus D)}{m_1(B_j)} + 2\eps\leq 3\eps.\\
\end{align*}
Define
\[d\nu(x,y) = \sum_{j\in J} \chi_{\tilde{B_j}}(x) \left(\frac{d|\mu|^1}{dm_1}(x)\right)^{-1} d\mu(x, y) + \chi_{(J_1\setminus \tilde{B})}\,d\mu(x,y),\] where $\tilde{B} = \bigcup_{j\in J} \tilde{B_j}$. It is clear that  $\frac{d|\nu|^1}{dm_1}(x) =1$ on $\tilde{B}$ and $\frac{d|\nu|^1}{dm_1}(x)\leq 1$ elsewhere. Note also that  for all $x\in J_1$,
\begin{align*}
\frac{d|\nu-\mu|^1}{dm_1}(x) &= \sum_{j\in J} \chi_{\tilde{B_j}}(x) \left(\left(\frac{d|\mu|^1}{dm_1}(x)\right)^{-1} -1\right) \frac{d|\mu|^1}{dm_1}\\
& \leq \frac{1}{1-\eps/8}-1 =\frac{\eps}{8-\eps}< \eps.
\end{align*}
Hence $T_\nu$ is a norm-one operator such that $\norm{T_\mu - T_\nu}<\eps$, $\|f_1- f_0\|< 3\eps$ and $\frac{d|\nu|^1}{dm_1}(x)=1$ for all $x\in {\rm supp}( f_1)$.
\end{proof}

\begin{lemma}\label{lem:complex2}
Let $0<\eps<1$. Suppose that $T_\nu$ is a norm-one operator in  $\mathcal{L}(L_1(m_1), L_1(m_2))$ and that $f$ is a nonnegative norm-one simple function in $S_{L_1(m_1)}$ satisfying $\|T_\nu f\| > 1-\frac{\eps^6}{2^7}$ and $\frac{d|\nu|^1}{dm_1}(x)=1$ for all $x$ in the support of $f$. Then there are a nonnegative simple function $\tilde{f}$ in $S_{L_1(m_1)}$ and a norm-one operator $T_{\tilde\nu}$ in $\mathcal{L}(L_1(m_1), L_1(m_2))$ such that
\[
\|T_{\tilde\nu}\tilde f\|=1,\quad \|T_\nu - T_{\tilde\nu}\|< 3\sqrt\eps\quad \text{and}\quad \|f-\tilde f\|< 3\eps.
\]
\end{lemma}

\begin{proof}
Let $f= \sum_{j=1}^n \beta_j \frac{\chi_{B_j}}{m_1(B_j)},$
where  $\{B_j\}_{j=1}^n$ are mutually disjoint Borel subsets of $J_1$, $\beta_j\geq 0$ and $m_1(B_j)>0$ for all $1\leq j\leq n$, and $\sum_{j=1}^n\beta_j=1$. Since $\|T_\nu f\| >1-\frac{\eps^6}{2^7}$ , there is  $g\in S_{L_\infty(m_2)}$ such that
\[
 1-\frac{\eps^6}{2^7}< \re \inner{ T_\nu f, g} = \sum_{j=1}^n \beta_j \re \int_{J_1\times J_2} \frac{\chi_{B_j}(x)}{m_1(B_j)}  (g(y))\, d\nu(x,y).
\]
Let $\displaystyle J=\left \{ j\in\{1,\ldots,n\}\, :\,\re \int_{J_1\times J_2} \frac{\chi_{B_j}(x)}{m_1(B_j)}g(y)\, d\nu(x,y) > 1-\frac{\eps^3}{2^6}\right\}$. From Lemma~\ref{elementary} it follows that
\[\sum_{j\in J}\beta_j >1-\frac{\eps^3}{2}.\]
Let $f_1 = \sum_{j\in J} \tilde \beta_j \frac{\chi_{B_j}}{m_1(B_j)}$, where $\tilde \beta_j = \beta_j / (\sum_{j\in J}\beta_j )$ for all $j\in J$. Then
 \[ \|f_1 - f\| \leq \nor{ \sum_{j\in J} (\tilde \beta_j - \beta_j) \frac{\chi_{B_j}}{m_1(B_j)} } + \sum_{j\in J} \beta_j \leq \eps^3< \eps.\]
Note that there is a Borel measurable function $h$ on $J_1\times J_2$ such that $d\nu(x,y) = h(x,y)\, d|\nu|(x,y)$ and $|h(x,y)|=1$ for all $(x,y)\in J_1\times J_2$. Let
\[ C= \left\{ (x,y) : |g(y)h(x,y)-1| < \frac{\sqrt\eps}{2^{3/2}} \right\}.\]
Define  two measures $\nu_{f}$ and $\nu_{c}$ as follows:
\begin{equation*}
\nu_f(A) = \nu(A\setminus C) \quad \text{and} \quad
\nu_c(A) = \nu(A\cap C)
\end{equation*}
for every Borel subset $A$ of $J_1\times J_2$. It is clear that
$$
d\nu= d\nu_{f} + d\nu_{c},\quad d|\nu_f|  =\bar  hd\nu_f,\quad  d|\nu_c| =\bar  hd\nu_c, \quad \text{and}\quad  d|\nu| = d|\nu_f|+d|\nu_c|.
$$
Since $\frac{d|\nu|^1}{dm_1}(x)=1$ for all $x\in \bigcup_{j=1}^n B_j$, we have
\[
1=\frac{d|\nu|^1}{dm_1}(x)  = \frac{d|\nu_{f}|^1}{dm_1}(x)+ \frac{d|\nu_{c}|^1}{dm_1}(x)
\]
for all $x\in B=\bigcup_{j=1}^n B_j$,
and we deduce that $|\nu|^1(B_j)=m_1(B_j)$ for all $1\leq j\leq n$.

We claim that $\frac{|\nu_{f}|^1(B_j) }{m_1(B_j)} \leq \frac{\eps^2}{2^2}$ for all $j\in J$. Indeed, if $|g(y)h(x,y)-1|\geq  \frac{\sqrt\eps}{2^{3/2}}$, then $\re (g(y)h(x,y))\leq 1-\frac{\eps}{2^4}$. So we have
\begin{align*}
1-\frac{\eps^3}{2^6} &\leq \frac{1}{m_1(B_j)}\re \int_{J_1\times J_2} \chi_{B_j(x)} g(y)\, d\nu(x,y)\\
&= \frac{1}{m_1(B_j)} \int_{J_1\times J_2} \chi_{B_j(x)} \re\big(g(y)h(x,y)\big)  \, d|\nu|(x,y)\\
&= \frac{1}{m_1(B_j)} \int_{J_1\times J_2} \chi_{B_j(x)} \re\big(g(y)h(x,y)\big)  \, d|\nu_f|(x,y) \\
&\ \ \ \ \ +  \frac{1}{m_1(B_j)} \int_{J_1\times J_2} \chi_{B_j(x)} \re\big(g(y)h(x,y)\big)  \, d|\nu_c|(x,y) \\
&\leq \frac{1}{m_1(B_j)} \left((1-\frac{\eps}{2^4}) |\nu_f|^1(B_j) +|\nu_c|^1(B_j)\right) \\
&= 1 - \frac{\eps}{2^4} \frac{|\nu_{f}|^1(B_j)}{m_1(B_j)}.
\end{align*} This proves our claim.

We also claim that for each $j\in J$, there exists a Borel subset $\tilde B_j$ of $B_j$  such that
\[ \left(1-\frac{\eps}2\right) m_1(B_j) \leq m_1(\tilde B_j) \leq m_1(B_j)\] and
\[ \frac{d|\nu_{f}|^1}{dm_1}(x)\leq \frac{\eps}2 \] for all $x\in \tilde B_j$.
Indeed, set $\tilde B_j = B_j\cap \left\{ x \in J_1 : \frac{d|\nu_{f}|^1}{dm_1}(x)\leq\frac {\eps}2\right\}$. Then
\[ \int_{B_j \setminus \tilde B_j} \frac{\eps}2 \, dm_1(x) \leq  \int_{B_j} \frac{d|\nu_{f}|^1}{dm_1}(x) \, dm_1(x) = |\nu_{f}^1|(B_j)\leq \frac{\eps^2}{2^2}m_1(B_j). \] This shows that $m_1(B_j \setminus \tilde B_j) \leq \frac{\eps}2 m_1(B_j)$. This proves our second claim.

Now, we define $\tilde{g}$ by $\tilde g(y) = \frac{g(y)}{|g(y)|}$ if $g(y)\neq 0$ and $\tilde g(y) = 1$ if $g(y)=0$, and we write $\tilde f = \sum_{j\in J} \tilde\beta_j \frac{\chi_{\tilde B_j} }{m_1(\tilde B_j) }$. Finally, we define the measure
\[
d\tilde \nu(x,y) = \sum_{j\in J} \chi_{\tilde B_j}(x)  \overline{\tilde g(y)} \overline{ h(x, y)} d\nu_c(x,y) \left( \frac{d|\nu_{c}|^1}{dm_1}(x) \right)^{-1} + \chi_{J_1\setminus \tilde B}(x) d\nu(x,y),
\]
where $\tilde B= \bigcup_{j\in J} \tilde B_j$. It is easy to see that $\frac{d|\tilde\nu|^1}{dm_1}(x) =1$ on $\tilde B$ and $\frac{d|\tilde\nu|^1}{dm_1}(x)\leq 1$ elsewhere. Note that
\begin{align*}
d(\tilde \nu- \nu)(x,y) &= \sum_{j\in J} \chi_{\tilde B_j}(x) \left[ \overline{\tilde g(y)} \overline{ h(x, y)}\left( \frac{d|\nu_{c}|^1}{dm_1}(x) \right)^{-1} -1\right] d\nu_c(x,y)   \\
&\ \ \ \ - \sum_{j\in J} \chi_{\tilde B_j}(x) d\nu_f(x,y).
\end{align*}
If $(x,y)\in C$, then $|g(y)|\geq 1-\frac{\sqrt{\eps}}{2^{3/2}} \geq 1-\frac{1}{2^{3/2}}$ and
\begin{align*}
 \left| \overline{\tilde g(y)} \overline{ h(x, y)}-1 \right| &=  \left|  \frac{g(y)}{|g(y)|}  h(x, y)-1  \right| \\
 & \leq \frac{ \left|  g(y) h(x, y)-1  \right| }{|g(y)|} + \frac{\big|1-|g(y)|\big|}{|g(y)|} \\
 &\leq 2 \frac{ \left|  g(y) h(x, y)-1  \right| }{|g(y)|} \leq 2\frac{\sqrt{\eps}}{2^{3/2}}\frac{2^{3/2}}{2^{3/2}-1}\leq 2\sqrt{\eps}.
\end{align*}
Hence, for all $(x,y)\in C$ we have
\begin{align*}
 \left| \overline{\tilde g(y)} \overline{ h(x, y)}\left( \frac{d|\nu_{c}|^1}{dm_1}(x) \right)^{-1} -1\right|  &\leq
 \left| \overline{\tilde g(y)} \overline{ h(x, y)}-1 \right|~ \left( \frac{d|\nu_{c}|^1}{dm_1}(x) \right)^{-1}  +   \left| \left( \frac{d|\nu_{c}|^1}{dm_1}(x) \right)^{-1} -1 \right|\\
 &\leq   2\sqrt{\eps}~ \left( \frac{d|\nu_{c}|^1}{dm_1}(x) \right)^{-1} +\left| \left( \frac{d|\nu_{c}|^1}{dm_1}(x) \right)^{-1} -1 \right|.
\end{align*}
So, we have for all $x\in J_1$,
\begin{align*}
\frac{d| \tilde \nu - \nu|^1}{dm_1}(x) &\leq \sum_{j\in J} \chi_{\tilde B_j}(x) \left[  2\sqrt{\eps}~ \left( \frac{d|\nu_{c}|^1}{dm_1}(x) \right)^{-1} +\left| \left( \frac{d|\nu_{c}|^1}{dm_1}(x) \right)^{-1} -1 \right|
   \right] \frac{d|\nu_c|^1}{dm_1}(x) \\
&\ \ \ \ + \sum_{j\in J} \chi_{\tilde B_j}(x) \frac{d|\nu_f|^1}{dm_1}(x)\\
&\leq \sum_{j\in J} \chi_{\tilde B_j}(x) \left( 2\sqrt{\eps} +   \left( 1-\frac{d|\nu_{c}|^1}{dm_1}(x)   \right)\right) + \sum_{j\in J} \chi_{\tilde B_j}(x) \left( \frac{d|\nu_{f}|^1}{dm_1}(x) \right)\\
&\leq 2\sqrt{\eps}+\eps <3\sqrt{\eps}.
\end{align*}
This gives that $\|T_\nu - T_{\tilde \nu} \|< 3\sqrt\eps$. Note also that, for all $j\in J$,
\begin{align*}
 \inner{ T_{\tilde \nu} \frac{\chi_{\tilde B_j}}{m_1(\tilde B_j)}, \tilde g} &= \int_{J_1\times J_2} \frac{\chi_{\tilde B_j}(x)}{m_1(\tilde B_j)} \tilde  g(y) \,d\tilde \nu(x,y)\\
 & =\int_{J_1\times J_2} \frac{\chi_{\tilde B_j}(x)}{m_1(\tilde B_j)} \overline{h(x,y)}\left( \frac{d|\nu_{c}|^1}{dm_1}(x) \right)^{-1} \,\, d\nu_c(x,y)\\
 &=\int_{J_1} \frac{\chi_{\tilde B_j}(x)}{m_1(\tilde B_j)}\left( \frac{d|\nu_{c}|^1}{dm_1}(x) \right)^{-1} \,  d|\nu_c|^1(x) \\
 &= \int_{J_1} \frac{\chi_{\tilde B_j}(x)}{m_1(\tilde B_j)} \, dm_1(x)=1.
 \end{align*}
Hence we get $\inner{T_{\tilde \nu} \tilde f, \tilde g} =1$, which implies that $\|T_{\tilde \nu} \tilde f\|=\|T_{\tilde \nu}\|= 1$. Finally,
\begin{align*}
\|\tilde f - f\| &\leq \|\tilde f - f_1 \| + \|f_1 - f\|\\
& = \nor{ \sum_{j\in J} \tilde \beta_j \frac{\chi_{\tilde B_j}}{m_1(\tilde B_j)} - \sum_{j\in J} \tilde \beta_j \frac{\chi_{B_j}}{m_1(B_j)}}  +\eps \\
& \leq \sum_{j\in J} \tilde \beta_j \left(\nor{ \frac{\chi_{\tilde B_j}}{m_1(\tilde B_j)} -  \frac{\chi_{ B_j}}{m_1(\tilde B_j)} } +\nor{ \frac{\chi_{ B_j}}{m_1(\tilde B_j)} - \frac{\chi_{B_j}}{m_1(B_j)}}\right) +\eps \\
&=2\sum_{j\in J}\tilde \beta_j \frac{ m_1(B_j\setminus \tilde B_j)}{m_1(\tilde B_j)} +\eps\\
&\leq 2\sum_{j\in J}\tilde \beta_j \frac{ \frac{\eps}2 m_1(B_j)}{m_1(\tilde B_j)} +\eps\leq  \frac{\eps}{1-\eps/2}+\eps < 3\eps. \qedhere
\end{align*}
\end{proof}

We are now ready to prove Theorem~\ref{thm:main1}.

\begin{proof}[Proof of Theorem~\ref{thm:main1}]
Let $0<\eps<1$. Suppose that $T$ is a norm-one element in $\mathcal{L}(L_1(m_1), L_1(m_2))$ and there is $f\in S_{L_1(m_1)}$ such that $\|Tf\|>1-\frac{\eps^{18}}{5^32^{27}}$. Then there is an isometric isomorphism $\psi$ from $L_1(m_1)$ onto itself such that $\psi(f)=|f|$. Using $T\circ \psi^{-1}$ instead of $T$, we may assume that $f$ is nonnegative. Since simple functions are dense in $L_1(m_1)$, we can choose a nonnegative simple function $f_0\in S_{L_1(m_1)}$ arbitrarily close to $f$ so that
$$
\|Tf_0\| >  1-\frac{\eps^{18}}{5^32^{27}} = 1-\frac{\eps_1^3}{2^6},
$$
where $\eps_1 = \frac{\eps^6}{5 \cdot2^7}$. By Lemma~\ref{lem:complex1}, there exist a norm-one bounded linear operator $T_\nu$ for some $\nu\in M(m_1, m_2)$ and a nonnegative simple function $f_1$ in $S_{L_1(M_1)}$ such that $\norm{T - T_\nu}< \eps_1$, $\|f_1- f\|<3\eps_1$ and $\frac{d|\nu|^1}{dm_1}(x)=1$ for all $x\in {\rm supp}( f_1)$. Then
\[
\|T_\nu f_1\| \geq \|Tf\| - \|Tf - T_\nu f\| - \|T_\nu(f-f_1)\|\geq 1-\frac{\eps_1^3}{2^6} - \eps_1-3\eps_1\geq 1-5\eps_1=1-\frac{\eps^6}{2^7}.
\]
Now, by Lemma~\ref{lem:complex2}, there exist a nonnegative simple function $\tilde{f}$ and an operator $T_{\tilde\nu}$ in $\mathcal{L}(L_1(m_1), L_1(m_2))$ such that $\|T_{\tilde\nu}\tilde f\|=\|T_{\tilde\nu}\|=1$, $\|T_\nu - T_{\tilde\nu}\|\leq 3\sqrt\eps$ and $\|f_1-\tilde f\|\leq 3\eps$. Therefore, $\|T-T_{\tilde \nu}\| < 4\sqrt\eps$ and $\|f-\tilde{f}\|< 4\eps$, which complete the proof.
\end{proof}

\section{The Bishop-Phelps-Bollob\'{a}s property
of $(L_1(\mu),L_\infty(\nu))$}\label{sec:L_1-L_infty}

Our aim now is to show that $(L_1(\mu), L_\infty(\nu))$ has the BPBp for any measure $\mu$ and any localizable measure $\nu$.

\begin{theorem}\label{thm:localizable}
Let $\mu$ be an arbitrary measure and let $\nu$ be a localizable measure. Then the pair $(L_1(\mu),L_\infty(\nu))$ has the BPBp. Moreover,
$$
\eta\bigl(L_1(\mu),L_\infty(\nu)\bigr)(\eps)\geq \left(\frac{\eps}{10}\right)^8 \qquad (0<\eps<1).
$$
\end{theorem}

By Corollaries \ref{cor:main1} and \ref{cor:to_L_infty-localizable}, it is enough to prove the result in the case where $\mu$ is  $\sigma$-finite and $\nu$ is the product measure on $[0,1]^\Lambda$. Therefore, we just need to prove the following result.

\begin{theorem}\label{thm:main2}
Assume $\mu$ is a $\sigma$-finite measure and $\nu$ is the product measure of Lebesgue measures on $[0,1]^\Lambda$. Let $0 <\eps
<1/3$, let $T:L_1(\mu)\longrightarrow  L_{\infty}(\nu)$ be a bounded linear operator of norm one and let $f_0\in S_{L_1(\mu)}$ satisfy $\|T(f_0)\|_{\infty}>1-\eps^8$. Then there exist $S\in \mathcal{L}(L_1(\mu), L_{\infty}(\nu))$ with $\|S\|=1$ and $g_0\in
S_{L_1(\mu)}$ such that
$$
\|S(g_0)\|_{\infty} =1,\quad \|T-S\|<2\eps
\quad \text{and} \quad \|f_0-g_0\|_1 < 10\eps.
$$
\end{theorem}

Recall that the particular case where $\Lambda$ reduces to one point was established in \cite{ACGM}. Actually, our proof is based on the argument given there.

Prior to giving the proof of Theorem~\ref{thm:main2}, we present the following representation result for operators from $L_1(\mu)$ into $L_\infty\bigl([0,1]^\Lambda\bigr)$ and one lemma.

Let $(\Omega, \Sigma, \mu)$ be a $\sigma$-finite measure space and let $K=[0,1]^\Lambda$ be the product space equipped with the product measure $\nu$ of the Lebesgue measures. Let $J$ be a countable subset of $\Lambda$ and let $\pi_{J}$ be the natural projection from $K$ onto $[0,1]^J$. Fix a sequence $(\Pi_n)$ of finite partitions of $[0,1]^J$ into sets of positive measure such that $\Pi_{n+1}$ is a refinement of $\Pi_n$ for each $n$, and the $\sigma$-algebra generated by $\bigcup_{n=1}^\infty \Pi_n$ is the Borel $\sigma$-algebra of $[0,1]^J$. For each $y\in K$ and $n\in \mathbb{N}$, let $B(n,\pi_J(y))$ be the set  in $\Pi_n$ containing $\pi_J(y)$. Then, given a Borel set $F$ of the form $F_0\times [0,1]^{\Lambda\setminus J}$ with $F_0 \subset [0,1]^J$, define
\[ \delta(F) = \left\{ y\in K : \lim_{n\to \infty} \frac{ \nu(F\cap \pi_J^{-1}(B(n,\pi_J(y)))}{\nu(\pi_J^{-1}(B(n,\pi_J(y)))}=1\right\}.\]
It is easy to check that $\delta(F) = \delta_J(F_0)\times [0,1]^{\Lambda\setminus J}$, where
\[\delta_J(F_0) = \left\{ y\in [0,1]^J : \lim_{n\to \infty} \frac{ \nu(\pi_J^{-1}(F_0\cap B(n,y)))}{\nu(\pi_J^{-1}(B(n,y)))}=1\right\}.\]
Using the martingale almost everywhere convergence theorem \cite{Doob}, we have
\[
\nu(F\Delta \delta(F))=0
\]
where $F\Delta \delta(F)$ denotes the symmetric difference of the sets $F$ and $\delta(F)$.

On the other hand, it is well-known that the space $\mathcal{L}(L_1(\mu), L_\infty(\nu))$ is isometrically isomorphic to the space $L_\infty(\mu\otimes \nu)$, where $\mu\otimes \nu$ denotes the product measure on $\Omega\times K$. More precisely, the operator $\widehat h$ corresponding to $h\in L_\infty(\mu\otimes\nu)$ is given by
\[ \widehat h(f)(t) = \int_\Omega h(\omega, t)f(\omega) \, d\mu(\omega)\] for $\nu$-almost every $t\in K$. For a reference, see \cite{DF}.

\begin{lemma}\label{basic}
Let $M$ be a measurable subset of $\Omega \times K$ with positive measure, $0<\eps <1$, and let $f_0$ be a simple function.
If $\|\widehat{\chi_M}(f_0)\|_{\infty}>1-\eps$, then there exists a simple function $g_0\in
S_{L_1(\mu)}$ such that
$$
\left\|[\widehat{\chi_M} + \widehat{\varphi}](g_0)\right\|_{\infty}=1 \quad \text{and} \quad
\|f_0-g_0\|_1<4\sqrt{\eps}
$$
for every simple function $\varphi$ in
$L_{\infty}(\mu\otimes \nu)$ with $\|\varphi\|_{\infty}\leq 1$ and vanishing on $M$.
\end{lemma}

\begin{proof}
Write $f_0=\sum_{j=1}^m \alpha_j\frac{\chi_{A_j}}{\mu(A_j)}\in S_{L_1(\mu)}$, where each $A_j$ is a
measurable subset of $\Omega$ with finite positive measure, $A_k\cap A_l = \emptyset$ for $k\neq l$, and $\alpha_j$ is a positive real number for every $j=1,\ldots, m$ with $\sum_{j=1}^{m} \alpha_j
=1$. Since $\|\widehat{\chi_M}(f_0)\|_{\infty}>1-\eps$, there is a measurable subset $B$ of $K$ such that $0<\nu(B)$ and
$$
\inner{\widehat{\chi_M}(f_0), \frac{\chi_B}{\nu(B)}} > 1-\eps.
$$
We may assume that there is a countable subset $J$ of $\Lambda$ such that $M=M_0\times [0,1]^{\Lambda\setminus J}$ and $B=B_0\times [0,1]^{\Lambda\setminus J}$ for some measurable subsets $M_0 \subset \Omega\times [0,1]^J$ and $B_0\subset [0,1]^J$. For each $j\in \{1,\ldots, m\}$, we write $M_j=M~\bigcap~ ( A_j\times B)=(M_0\cap (A_j\times B_0))\times[0,1]^{\Lambda\setminus J}$ and define
$$
H_j=\{(x,y) \,:\, x\in A_j, y\in \delta\big((M_j)_x\big)
\}.
$$
As in the proof of \cite[Proposition~5]{PS}, the $H_j$'s are disjoint measurable subsets of $\Omega \times K$. We note
that for each $j\in \{1,\ldots, m\}$, we have $H_j\subset  A_j\times \delta(B)$ and $(\mu\otimes \nu)(M_j\Delta H_j) =0$.

Now, by Fubini theorem, we have that
\begin{align*}
1-\eps &< \langle \widehat{\chi}_M(f_0), \frac{\chi_B}{\nu(B)}\rangle\\
&=\sum_{j=1}^m \frac{\alpha_j}{\mu(A_j)\nu(B)} \int_{\Omega\times K} \chi_{M_j}(x,y)~d(\mu\otimes \nu)\\
&=\sum_{j=1}^m \frac{\alpha_j}{\mu(A_j)\nu(B)} \int_{\Omega\times K} \chi_{H_j}(x,y)~d(\mu\otimes \nu)\\
&=\sum_{j=1}^m \frac{\alpha_j}{\nu(B)} \int_{\delta(B)} \frac{\mu(H_j^y)}{\mu(A_j)}~d\nu(y)\\
&=\frac1{\nu(\delta(B))} \int_{\delta(B)} \sum_{j=1}^m  {\alpha_j}\frac{\mu(H_j^y)}{\mu(A_j)}~d\nu(y).
\end{align*}
So, there exists $y_0\in \delta(B)$ such that
\[
\sum_{j=1}^m  {\alpha_j}\frac{\mu(H_j^{y_0})}{\mu(A_j)}>1-\eps.
\]
Let $J = \left\{ j\in\{1,\ldots,m\}\, : \, \frac{\mu(H_j^{y_0})}{\mu(A_j)}>1-\sqrt \eps\right\}$. For each $j\in J$, we have that $\mu(A_j\setminus H_j^{y_0}) < \sqrt\eps \mu(A_j)$ and, by Lemma~\ref{elementary}, we also have $\alpha_J := \sum_{j\in J} \alpha_j >1-\sqrt{\eps}$. Define
$$
g_0 = \sum_{j\in J} \beta_j \frac{\chi_{H_j^{y_0}}}{\mu(H_j^{y_0})},
$$
where $\beta_j={\alpha_j}/\alpha_J$. Then
\begin{align*}
\|g_0- f_0\| &< \nor{ \sum_{j\in J}  \beta_j \frac{\chi_{H_j^{y_0}}}{\mu(H_j^{y_0})}  - \sum_{j\in J} \alpha_j \frac{\chi_{A_j}}{\mu(A_j)} } +\sqrt{\eps}\\
& \leq \nor{  \sum_{j\in J}  \beta_j \frac{\chi_{H_j^{y_0}}}{\mu(H_j^{y_0})}  - \sum_{j\in J} \beta_j \frac{\chi_{A_j}}{\mu(A_j)} } + \nor {  \sum_{j\in J} \beta_j \frac{\chi_{A_j}}{\mu(A_j)} - \sum_{j\in J} \alpha_j \frac{\chi_{A_j}}{\mu(A_j)} } + \sqrt\eps\\
&\leq  \nor{  \sum_{j\in J}  \beta_j \frac{\chi_{H_j^{y_0}}}{\mu(H_j^{y_0})}  - \sum_{j\in J} \beta_j \frac{\chi_{H_j^{y_0}}}{\mu(A_j)} } + \nor{ \sum_{j\in J} \beta_j \frac{\chi_{H_j^{y_0}}}{\mu(A_j)} -\sum_{j\in J} \beta_j \frac{\chi_{A_j}}{\mu(A_j)}  }  +2\sqrt\eps\\
&\leq 2\frac{\mu(A_j\setminus H_j^{y_0})}{\mu(A_j)} + 2\sqrt{\eps} \leq 4\sqrt\eps
\end{align*}
We claim that $\widehat{\chi_M}+\widehat{\varphi}$ attains its norm at $g_0$.
Let $B_n = \pi_J^{-1}(B(n,\pi_J({y_0})))$ for each $n$. Note that for every $x\in H_j^{y_0}$ we have $(x,{y_0})\in H_j$, which implies
that
$$
\lim_{n\to \infty} \frac{\nu\big((M_j)_x\cap B_n\big)}{\nu(B_n)} =1.
$$
It follows from the Lebesgue dominated
convergence theorem and Fubini theorem that, for each $j\in J$,
\begin{eqnarray*} 1 & = &\lim_{n\to \infty} \frac{1}{\mu(H_j^{y_0})} \int_{H_j^{y_0}} \frac{\nu\big((M_j)_x\cap
B_n\big)}{\nu(B_n)} ~d\mu(x)\\ & =& \lim_{n\to \infty} \frac{(\mu \otimes \nu)(M_j\cap (H_j^{y_0}\times
B_n))}{\mu(H_j^{y_0}) \nu(B_n)}.
\end{eqnarray*}
On the other hand, since the simple function $\varphi$ is assumed to vanish on $M$ and $\|\varphi\|_{\infty} \leq 1$, we have
\begin{eqnarray*}
\left|\big\langle \widehat{\varphi}\big( \frac{\chi_{H_j^{y_0}}}{\mu(H_j^{y_0})}\big),
\frac{\chi_{B_n}}{\nu(B_n)}\big\rangle\right|& =& \left| \frac{1}{\mu(H_j^{y_0}) \nu(B_n)}\int_{H_j^{y_0}\times
B_n} \varphi ~d(\mu\otimes \nu)\right| \\ & \leq &  \frac{(\mu\otimes \nu)((H_j^{y_0}\times B_n)\setminus
M_j)}{\mu(H_j^{y_0}) \nu(B_n)}\\ & = & 1- \frac{(\mu \otimes \nu)(M_j\cap (H_j^{y_0}\times B_n))}{\mu(H_j^{y_0})
\nu(B_n)} \longrightarrow 0,
\end{eqnarray*} as
$n\to \infty$. Therefore,
\begin{eqnarray*}
1\geq  \big\|[\widehat{\chi_M} + \widehat{\varphi}](g_0)\big\|_{\infty} &\geq & \lim_{n\to \infty}
\left|\inner{(\widehat{\chi_M} + \widehat{\varphi})\Big(\sum_{j\in J} \beta_j
\frac{\chi_{H_j^{y_0}}}{\mu(H_j^{y_0})}\Big), \frac{\chi_{B_n}}{\nu(B_n)}}\right|\\
&= & \lim_{n\to \infty} \sum_{j\in J}
\beta_j \frac{(\mu \otimes
\nu)(M\cap (H_j^{y_0}\times B_n))}{\mu(H_j^{y_0}) \nu(B_n)}\\
&\geq&
\lim_{n\to \infty} \sum_{j\in J}
\beta_j \frac{(\mu \otimes\nu)(M_j\cap (H_j^{y_0}\times B_n))}{\mu(H_j^{y_0}) \nu(B_n)}= 1,
\end{eqnarray*}
which shows that $\widehat{\chi_M}+\widehat{\varphi}$ attains its norm at $g_0$.
\end{proof}

We are now ready to present the proof of the main result in this section.

\begin{proof}[Proof of Theorem~\ref{thm:main2}] Since the set of all simple functions is dense in
$L_1(\mu)$, we may assume $$f_0=\sum_{j=1}^m
\alpha_j\frac{\chi_{A_j}}{\mu(A_j)}\in S_{L_1(\mu)},$$ where each
$A_j$ is a measurable subset of $\Omega$ with finite positive
measure, $A_k\cap A_l = \emptyset$ for $k\neq l$, and every $\alpha_j$ is a nonzero complex number with $\sum_{j=1}^{m} |\alpha_j| =1$.
We may also assume that $0<\alpha_j\leq 1$ for every
$j=1,\ldots, m$. Indeed, there exists an isometric isomorphism $\Psi : L_1(\mu) \longrightarrow L_1(\mu)$ such that $\Psi(f_0)=|f_0|$. Hence we may replace $T$ and $f_0$ by
$T\circ \Psi^{-1}$ and $\Psi(f_0)$, respectively.

Let $h$ be the element in $L_{\infty}(\Omega\times K, \mu\otimes
\nu)$ with $\|h\|_{\infty}=1$ corresponding to $T$, that is, $T= \widehat{h}$. We may find a simple function
$$
h_0\in L_{\infty}(\Omega\times K, \mu\otimes \nu),
\quad \|h_0\|_{\infty} =1
$$
such that $\|h-h_0\|_{\infty} < \|T(f_0)\|_{\infty} -
(1-\eps^8)$, hence $\|\widehat{h}_0(f_0)\|_{\infty} > 1-\eps^8$. We can write $h_0= \sum_{l=1}^p c_l \chi_{D_l}$, where each $D_l$ is a measurable subset of $\Omega\times K$ with positive measure, $D_k\cap D_l= \emptyset$ for $k\neq l$, $|c_l|\leq 1$ for every $l=1, \ldots, p$, and $|c_{l_0}| =1$ for
some $1\leq l_0\leq p$.

Let $B$ be a Lebesgue measurable subset of $K$ with $0<\nu(B)<\infty$ such that
$$\left|\inner{\widehat{h}_0(f_0), \frac{\chi_B}{\nu(B)}} \right| > 1-\eps^8.$$
Choose $\theta\in \mathbb{R}$ so that
\begin{eqnarray*}
1- \eps^8 &< & \big|\langle\widehat{h}_0(f_0),
\frac{\chi_B}{\nu(B)}\rangle\big|\\ &=& \e^{i\theta} \langle\widehat{h}_0(f_0),
\frac{\chi_B}{\nu(B)}\rangle\\
&=& \sum_{j=1}^m\alpha_j~ \e^{i \theta}
\inner{\widehat{h}_0(\frac{\chi_{A_j}}{\mu(A_j)}),
\frac{\chi_B}{\nu(B)}}.
\end{eqnarray*}
Set
$$
J=\left\{j\in \{1,\ldots,m\}\, : \, \re \big[~ \e^{i \theta}
\langle\widehat{h}_0(\frac{\chi_{A_j}}{\mu(A_j)}),
\frac{\chi_B}{\nu(B)}\rangle\big] ~> 1-\eps^4\right\}.
$$
By Lemma~\ref{elementary}, we have
$$\alpha_J= \sum_{j\in J} \alpha_j
> 1-\frac{\eps^8}{1-(1-\eps^4)}=1-\eps^4.
$$
We define $$f_1=\sum_{j\in J} \left(\frac{\alpha_j}{\alpha_J} \right)\frac{\chi_{A_j}}{\mu(A_j)}.$$
We can see that $\|f_1\|_1 =1$,
\begin{eqnarray*}
\|f_0 - f_1\|_1 &\leq & \Big\|\sum_{j\notin J}\alpha_j
\frac{\chi_{A_j}}{\mu(A_j)}\Big\|_1
 +\Big (\frac{1}{\alpha_J} -1\Big)~\Big\|\sum_{j\in J}\alpha_j
\frac{\chi_{A_j}}{\mu(A_j)}\Big\|_1\\
&=& \sum_{j\notin J}\alpha_j + (1-\alpha_J) = 2 (1-\alpha_J) < 2\eps^4.
\end{eqnarray*}
and
\begin{eqnarray*}
\left|\langle\widehat{h}_0(f_1), \frac{\chi_B}{\nu(B)}\rangle \right| &\geq
&\re \left[
\e^{i\theta}\inner{ \widehat{h}_0(f_1), \frac{\chi_B}{\nu(B)}}\right]\\
&=& \frac{1}{\alpha_J} \sum_{j\in J} \alpha_j~\re \left[
\e^{i\theta}\inner{ \widehat{h}_0(\frac{\chi_{A_j}}{\mu(A_j)}), \frac{\chi_B}{\nu(B)}}\right]\\
&>& \frac{1}{\alpha_J} \sum_{j\in J} \alpha_j (1-\eps^4) =
1-\eps^4.
\end{eqnarray*}
Let $L  = \left\{ l\in\{1,\ldots,p\}\, :\, \re(\e^{i\theta} c_l) > 1-\frac{\eps^2}{2}\right\}$. On the other hand, for each $j\in J$, we have
\begin{eqnarray*}
1-\eps^4 & < & \re \left[
\e^{i\theta}\inner{\widehat{h}_0(\frac{\chi_{A_j}}{\mu(A_j)}),
\frac{\chi_B}{\nu(B)}}\right]\\
&=& \sum_{l=1}^p\re(\e^{i\theta} c_l)~ \frac{(\mu\otimes \nu) (D_l\cap
(A_j\times B))}{\mu(A_j)\nu(B)} \\
&\leq& \sum_{l\in \{1, \dots, p\}\setminus L} (1-\frac{\eps^2}{2} )\frac{(\mu\otimes \nu) (D_l\cap
(A_j\times B))}{\mu(A_j)\nu(B)}\\
&\ &\ \ \ \ \ + \sum_{l\in L} \frac{(\mu\otimes \nu) (D_l\cap
(A_j\times B))}{\mu(A_j)\nu(B)}\\
&\leq& 1 -  \frac{\eps^2}{2}\sum_{l\in \{1, \dots, p\}\setminus L}  \frac{(\mu\otimes \nu) (D_l\cap
(A_j\times B))}{\mu(A_j)\nu(B)}.
\end{eqnarray*}
This implies that for each $j\in J$
\[
\sum_{l\in \{1, \dots, p\}\setminus L}  \frac{(\mu\otimes \nu) (D_l\cap
(A_j\times B))}{\mu(A_j)\nu(B)}\leq 2\eps^2.
\]
Since
\[\sum_{l=1}^p  \frac{(\mu\otimes \nu) (D_l\cap
(A_j\times B))}{\mu(A_j)\nu(B)}>1-\eps^4,
\]
for every $j\in J$ we have that
\[ \sum_{l\in L}\frac{(\mu\otimes \nu)(D_l\cap (A_j\times B))}{\mu(A_j)\nu(B)}\geq (1-\eps^4-2\eps^2)\geq 1-3\eps^2.\]
Set $D= \bigcup_{l\in L} D_l$. Then we can see
\begin{eqnarray*} \langle\widehat{\chi}_{D} (f_1),
\frac{\chi_B}{\nu(B)}\rangle &=& \sum_{j\in J} \big(\frac{\alpha_j}{\alpha_J}\big) \cdot
\sum_{l\in L}
\frac{\mu\otimes \nu(D_l\cap (A_j\times B))}{\mu(A_j)\nu(B)}\geq 1-3\eps^2.
\end{eqnarray*}
By Lemma~\ref{basic}, there is $g_0\in S_{L_1(\mu)}$ such that $\|(\widehat{\chi}_{D} + \widehat{\varphi})(g_0)\|_{\infty} =1$ and $\|f_1 -g_0\| <4\sqrt{3\eps^2}<8\eps$ for every simple function $\varphi$ in $L_{\infty}(\mu\otimes \nu)$ vanishing on $D$ with $\|\varphi\|_{\infty}\leq 1$. Therefore, we have
$$
\|f_0-g_0\|_1\leq
\|f_0-f_1\|_1 + \|f_1-g_0\|_1 \leq 2\eps^4 + 8\eps<10\eps.
$$
Define $$h_1 = \e^{-i\theta}~\chi_{D} + \sum_{l\notin L} c_l~
\chi_{D_l} \in L_{\infty}(\mu\otimes \nu).$$ Let $S$ be the operator in $
\mathcal{L}(L_1(\mu), L_{\infty}(m))$ corresponding to $h_1$. Then we get
$$
\|S(g_0)\|_{\infty}= \|\widehat{h_1}(g_0)\|_{\infty}=1
$$
and
$$
\|h_0 -h_1\|_{\infty} =  \max_{l\in L} |c_l -
\e^{-i\theta}| = \max_{l\in L} |\e^{i\theta} c_l - 1|.
$$
As $\re
(\e^{i\theta} c_l) > 1- \frac{\eps^2}{2}$ for every $l\in L$, we have that
\begin{eqnarray*}
\big(\im (\e^{i\theta} c_l)\big)^2 &\leq& 1 - \big(\re (\e^{i\theta} c_l)\big)^2 < 1-(1-\frac{\eps^2}{2})^2 = \eps^2 - \frac{\eps^4}{4}.
\end{eqnarray*}
Since
\begin{eqnarray*}
|\e^{i\theta} c_l - 1|  &= &\sqrt{\big(1- \re ( \e^{i\theta} c_l)\big)^2 + \big(\im
(\e^{i\theta} c_l)\big)^2}\\ & < & \sqrt{\eps^4/4 + (\eps^2 - \eps^4/4)}
=\eps,
\end{eqnarray*}
we conclude that
$$
\|h_0 - h_1\|_{\infty} < \eps
$$
and
\begin{equation*}
\|T-S\|_{\infty}\leq  \|h - h_0\|_{\infty} + \|h_0 - h_1\|_{\infty} < \eps^8 + \eps < 2\eps.\qedhere
\end{equation*}
\end{proof}

\section{The Bishop-Phelps-Bollob\'as Property for some operators from $L_1(\mu)$ into $C(K)$}\label{sec:L_1-C(K)}

Throughout this section, we consider only a finite measure $\mu$ on a measurable space $(\Omega, \Sigma)$ and \textbf{real} Banach spaces $L_1(\mu)$ and $C(K)$. Our aim is to obtain the Bishop-Phelps-Bollob\'{a}s property for some classes of operators from $L_1(\mu)$ to $C(K)$, sharpening the results about denseness of norm-attaining operators given by Iwanik in 1982 \cite{Iwa2}.

We use the following standard representation of operators into $C(K)$ \cite[Theorem~1 in p.~490]{DS}.

\begin{lemma} Given a bounded linear operator $T:X\longrightarrow C(K)$, define $F : K\longrightarrow X^*$ by $F(s) = T^*(\delta_s)$, where $\delta_s$ is the point measure at $s\in K$. Then, for $x\in X$, the relation
$Tx(s) = \inner{x, F(s)}$ defines an isometric isomorphism of $\mathcal{L}(X, C(K))$ onto the space of weak$^*$ continuous functions from $K$ to $X^*$ with the supremum norm. Moreover,  compact operators correspond to norm continuous functions.
\end{lemma}

Iwanik \cite{Iwa2} considered operators $T\in \mathcal{L}(L_1(\mu),C(K))$ satisfying one of the following conditions:
\begin{enumerate}
  \item The map  $s\longmapsto T^*\delta_s$ is continuous in measure.
  \item There exists a co-meager set $G\subset K$ such that $\{T^*\delta_s:s\in G\}$ is norm separable in $L_\infty(\mu)$.
\end{enumerate}
We recall that a subset $A$ is said to be a co-meager subset of $K$ if the set $K\setminus A$ is meager, that is, of first category.

\begin{theorem} \label{KIwa}
Let $0<\eps<1$. Suppose that $T\in \mathcal{L}(L_1(\mu),C(K))$ (real case) has norm one and satisfies condition (1). If $\|Tf\|>1-\frac{\eps^2}6$ for some $f\in S_{L_1(\mu)}$, then there exist $S\in \mathcal{L}(L_1(\mu),C(K))$ with $\|S\|=1$ and $g\in S_{L_1(\mu)}$ such that $\|Sg\|=1$, $\|S-T\|<\eps$, and $\|f-g\|<\eps$. Moreover,  $S$ also satisfies condition (1).
\end{theorem}

\begin{proof}
Without loss of generality, we assume that there exists $s_0\in K$ such that
$$
Tf (s_0)>1-\frac{\eps^2}6.
$$
Consider the function $G:L_\infty(\mu)\longrightarrow L_\infty(\mu)$ given by
$$
G(h)= \Big(h \wedge (1-\eps/3)\Big) \vee (-1+\eps/3) \qquad \bigl(h\in L_\infty(\mu)\bigr).
$$
Since the lattice operation $G$ is continuous in the $L_\infty$ norm and $T$ satisfies condition (1), we can see that the mapping $s\longmapsto GT^*\delta_s$ is continuous in measure, hence weak$^*$-continuous. Let $\bar S$ be the element of  $\mathcal{L}(L_1(\mu),C(K))$ represented by the function $F(s) := GT^*\delta_s$.  Then
\[ \| \bar S - T \| = \sup_{s\in K}\| F(s) - T^*\delta_s \| \leq \frac{\eps}3.
\]
Let
$$
C=\left\{\omega \in \Omega ~:~ \sign\big(f(\omega)\big)T^*\delta_s(\omega) > 1-\frac{\eps}3\right\}
$$
and define $S= \overline{S}/\|\overline{S}\|$ and $g=f|_C/\|f|_C\|$,  where $f|_C$ is the restriction of $f$ to the subset $C$. It is easy to see that $S$ satisfies condition (1) and
$$
\|S-T\|\leq\|S-\overline{S}\|+\|\overline{S}-T\| = | \|\bar S\|-1| + \|\bar S - T\| \leq 2\|\bar S - T\|< \eps.
$$
Moreover, we get
\begin{eqnarray*}
 1-\frac{\eps^2}6
 &<&Tf (s_0)=\langle T^*\delta_{s_0}, f \rangle=\int_\Omega T^*\delta_{s_0}(\omega) f(\omega)\, d\mu\\
 &=&\int_{C} \sign\big(f(\omega)\big)T^*\delta_{s_0}(\omega) |f(\omega)|\, d\mu+\int_{\Omega\backslash C}  \sign\big(f(\omega)\big)T^*\delta_{s_0}(\omega) |f(\omega)|\, d\mu\\
 &\leq& \int_{C} |f(x)|\,d\mu+(1- \frac{\eps}3)\int_{\Omega\backslash C}|f(\omega)|\, d\mu\\
 &=&1-\frac{\eps}3\int_{\Omega\backslash C}|f(x)|d\mu,
 \end{eqnarray*}
which implies that
$$
\int_{\Omega\backslash C}|f(x)|\,d\mu< \frac{\eps}2.
$$
Therefore,
\begin{eqnarray*}
  \|g-f\|
 &\leq& \|g-f|_C\|+\|f|_C-f\|=2(1-\|f_C\|)\\
 &=&2\int_{\Omega\backslash C}|f(x)|d\mu< \eps
\end{eqnarray*}
On the other hand, we see that $Sg(s_0)=\langle S^* \delta_{s_0}, g \rangle = 1$ because $S^*\delta_{s_0}(\omega) = \sign\big(f(x)\big)=\sign\big(g(\omega)\big)$ for every $\omega \in C$. This completes the proof.
\end{proof}

We do not know, and it is clearly of interest, for which topological compact Hausdorff spaces $K$ all operators in $\mathcal{L}(L_1(\mu),C(K))$ satisfy condition (1).

We recall that a bounded linear operator $T$ from $L_1(\mu)$ into a Banach space $X$ is said to be \emph{Bochner representable} if there is a bounded strongly measurable function $g:\Omega \longrightarrow X$ such that
\[
Tf = \int f(\omega)g(\omega) \, d\mu(\omega) \qquad \bigl(f\in L_1(\mu)\bigr).
\]
The Dunford-Pettis-Phillips Theorem \cite[Theorem~12 p.~75]{DU} says that $T\in \mathcal{L}(L_1(\mu), X)$ is weakly compact if and only if $T$ is Bochner representable by a function $g$ which has an essentially relatively weakly compact range. Iwanik \cite{Iwa2} showed that every Bochner representable operator from $L_1(\mu)$ into $C(K)$ satisfies condition (1). Moreover, we get the following.

\begin{corollary}Let  $0<\eps <1$. Suppose that $T\in \mathcal{L}(L_1(\mu),C(K))$ (real case) has norm-one and it is Bochner representable (resp.\ weakly compact). If  $\|Tf\|>1-\frac{\eps^2}6$ for some $f\in S_{L_1(\mu)}$, then there exist a Bochner representable (resp.\ weakly compact) operator $S\in \mathcal{L}(L_1(\mu),C(K))$ with $\|S\|=1$ and $g\in S_{L_1(\mu)}$ such that $\|Sg\|=1$, $\|S-T\|<\eps$, and $\|f-g\|<\eps$.
\end{corollary}

\begin{proof}
By Theorem~\ref{KIwa}, it is enough to show that if $T$ is a Bochner representable operator from $L_1(\mu)$ into $C(K)$, then $F(s) = T^*\delta_s$ is continuous in measure and that the operator $S$ defined in the proof is Bochner representable.

Let $g:\Omega \longrightarrow C(K)$ be a bounded strongly measurable function which represents $T$. It is easy to check that $F(s)= g(\cdot )(s)$ for all $s\in K$.
Since the range of $g$ is separable, the range of $T$ is separable and contained in a separable  sub-algebra $A$ of $C(K)$ with unit. By the Gelfand representation theorem, $A$ is isometrically isomorphic to $C(\bar K)$ for some compact metrizable space $\bar K$. So, we may assume that $K$ is metrizable. To show that the mapping $F(s) = T^*\delta_s=g(\omega)(s)$ is continuous in measure, assume that a sequence $(s_n)$ converges to $s$ in $K$. Then for all $\omega\in \Omega$,
\[ \lim_{n\to \infty} |g(\omega)(s_n)-g(\omega)(s)|=0.\]
By the dominated convergence theorem, we have that
\[ \lim_{n\to \infty} \sup_{f\in S_{L_\infty(\mu)}}  \int f(\omega) (g(\omega)(s_n) - g(\omega)(s)) \, d\mu(\omega) \leq \lim_{n\to \infty} \int |g(\omega)(s_n) - g(\omega)(s)| \, d\mu(\omega) =0.\]
Hence the sequence $(g(\cdot)(s_n))_n$ converges to $g(\cdot)(s)$ in measure. That is, $(F(s_n))_n$ converges to $F(s)$.

We note that  the operator $\bar S$ in the proof of Theorem~\ref{KIwa} is determined by $GT^*\delta_s = G(g(\cdot )(s))$. Since the mapping
\[ s\longmapsto G(g(\cdot)(s))(\omega) = (g(\omega)(s) \wedge (1-\eps/3))\vee(-1+\eps/3))
\]
is continuous for each $\omega\in \Omega$, the operator $\bar S$ is Bochner representable by this mapping. Finally, if $T$ is weakly compact, then the proof is done by the Dunford-Pettis-Phillips theorem.
\end{proof}

As observed in \cite{Iwa2}, the operator $T: L_1[0,1]\longrightarrow C[0,1]$ determined by $T^*\delta_s = \chi_{[0,s]}$ is not Bochner representable, but satisfies condition (1).

For condition (2), we have the following result.

\begin{theorem}\label{thm:iwanik2}
Let $0<\eps<1$. Suppose that $T\in \mathcal{L}(L_1(\mu),C(K))$ (real case) has norm-one and satisfies condition (2).  If $\|Tf\|>1-{{\eps^2}\over{4}}$ for some $f\in S_{L_1(\mu)}$, then there exist $S\in \mathcal{L}(L_1(\mu),C(K))$ with $\|S\|=1$ and $g\in S_{L_1(\mu)}$ such that $\|Sg\|=1$, $\|S-T\|<\eps$, and $\|f-g\|<\eps$. Moreover, $S$ also satisfies condition (2).
\end{theorem}

\begin{proof}
By using a suitable isometric isomorphism, we may first assume that $f$ is nonnegative. Let $G$ be the co-meager set in the condition (2) and $(T^*\delta_{s_k})_k$ be a sequence which is $\|\cdot \|_\infty$-dense in the closure of $\{T^*\delta_s~:~s\in G\}\subset L_\infty(\mu)$. Observe that the sets
\[ \{\omega \in \Omega: a < T^*\delta_{s_k} (\omega) <b \}\]
where $a, b\in \mathbb{Q}$ and $k\geq 1$, form a countable family $\{A_i\}_i$ of measurable subsets of $\Omega$. We define, for each $i$, the functions
\[
u_i(s) = {\rm ess.}\inf\{ T^*\delta_s(\omega) : \omega \in A_i\} \quad \text{and} \quad v_i(s) = {\rm ess.}\sup\{ T^*\delta_s(\omega) : \omega \in A_i\}.
\]
Let $U_i$ and $V_i$ be the set of all continuity points of $u_i$ and $v_i$ for all $i$, respectively. Let $F$ be the intersection of all subsets $U_i$'s and $V_i$'s. We claim that the functions $u_i$'s are upper semi-continuous and the functions $v_i$'s are lower semi-continuous. Indeed, recall that
\[ v_i(s) = \inf\Big\{ \lambda \in \mathbb{R} : \mu\{ \omega \in A_i: T^*\delta_s(\omega)>\lambda\} =0\Big\},\] where $\inf \emptyset = \infty$ and $\inf \mathbb{R} = -\infty$. To show that the set
$\{ s : \lambda< v_i(s)\}$ is open in $K$ for all $\lambda\in \mathbb{R}$, suppose that $v_i(s_0)>\lambda_0$ for some $s_0\in K$ and $\lambda_0\in \mathbb{R}$. It suffice to prove that there is an open neighborhood $V$ of $s_0$ such that $V\subset \{ s : v_i(s) >\lambda_0\}$. We note that
$\mu\{ \omega \in A_i: T^*\delta_{s_0}(\omega) >\lambda_0\} >0$ and there exists $\lambda_1>\lambda_0$ such that
\[ \mu\{ \omega\in A_i : T^*\delta_{s_0}(\omega) >\lambda_1\} >0.\]
Let $E = \{ \omega\in A_i : T^*\delta_{s_0}(\omega) >\lambda_1\} $. Then
\[ \frac{1}{\mu(E)}\int_E T^*\delta_{s_0}(\omega) d\mu(\omega) > \lambda_1>\lambda_0.\]
Since the map $s\longmapsto T^*\delta_s$ is weak$^*$ continuous on $L_\infty(\mu)$, the set
\[ V:= \left \{ s\in K :  \frac{1}{\mu(E)}\int_E T^*\delta_{s}(\omega) d\mu(\omega)  >\lambda_1    \right\}\]
is an open subset containing $s_0$. We note that  $\mu\{ \omega\in A_i : T^*\delta_s(\omega) >\lambda_1\}>0$ for all $s\in V$. Otherwise, there is $s_1\in V$ such that $\mu\{ \omega \in A_i : T^*\delta_{s_1}(\omega) >\lambda_1\}=0$. Then
$T^*\delta_{s_1}(\omega)\leq \lambda_1$ almost everywhere $\omega\in A_i$ and
\[  \frac{1}{\mu(E)}\int_E T^*\delta_{s_1}(\omega) d\mu(\omega)  \leq \lambda_1.\] This is a contradiction to the fact that $s_1$ is an element of $V$, which implies that $v_i(s)>\lambda_0$ for all $s\in V$ and $V\subset \{ s : v_i(s) >\lambda_0\}$. This gives the lower semi-continuity of $v_i$. The upper semi-continuity of $u_i$ follows from the fact that $-u_i$ is lower semi-continuous. The claim is proved.

We deduce then that the set $F$ is co-meager (c.f.\ see \cite[\S~32~II. p.~400]{Kur}). Since the set $\{s\,:\, s\in K ,\, |Tf(s)| > 1-\frac{\eps^2}4\}$ is nonempty and open, there exists $s_0\in F\cap G$ such that  $|Tf(s_0)|>1-{{\eps^2}\over{4}}$. Without loss of generality, we may assume that
$$Tf(s_0) = \inner{T^*\delta_{s_0}, f}>1-{{\eps^2}\over{4}}.$$
Because of the denseness of the sequence $(T^*\delta_{s_k})_k$, there exists $k_0\in \mathbb{N}$ such that
\[Tf(s_{k_0}) = \inner{T^*\delta_{s_{k_0}}, f}>1-{{\eps^2}\over{4}}  \ \ \ \ \text{ and } \ \ \
\|T^* \delta_{s_0} - T^*\delta_{s_{k_0}} \| < \frac{\eps}{4}.\]
Fix $q\in \mathbb{Q}$ such that $1-\frac{3}{4}\eps<q<1-\frac{\eps}2$ and let $$C=\left\{\omega\in \Omega~:~T^*\delta_{s_{k_0}}(\omega)> q\right\}.$$
Then \begin{align*}
1-{{\eps^2}\over{4}}
<&\inner{T^*\delta_{s_{k_0}}, f}=\int_\Omega T^*\delta_{s_{k_0}}(\omega)f(\omega)\,d\mu\\
=&\int_{C} T^*\delta_{s_{k_0}}(\omega)f(\omega)\,d\mu+\int_{\Omega\setminus C} T^*\delta_{s_{k_0}}(\omega)f(\omega)\,d\mu\\
\leq&\int_{C}f(\omega)\,d\mu+\left(1-{{\eps}\over{2}}\right)\int_{\Omega\setminus C}f(\omega)\,d\mu\\
=&1-{{\eps}\over{2}}\int_{\Omega \setminus C}f(\omega)\,d\mu.
\end{align*}
Hence we have that
\[
\int_{\Omega \setminus C} f(\omega) \, d\mu < \frac{\eps}2\quad  \text{and}\quad \int_{C} f(\omega) \, d\mu > 1-\frac{\eps}2.
\]
Let $B_n = \{ \omega \,:\, q<T^*\delta_{s_{k_0}}(\omega) <n\}$ for each $n$. Then $C=\bigcup_{n=1}^\infty B_n$ and  there exists $n_0$ such that
\[ \int_{B_{n_0}}  f(\omega) \, d\mu > 1-\frac{\eps}2.\]
Hence $B_{n_0} = A_{i_0}$ for some ${i_0}$ and $\mu(A_{i_0})>0$. This implies that  $u_{i_0}(s_{k_0})\geq q$ and $u_{i_0}(s_0) \geq q-\frac{\eps}{4}>1-\eps$. Setting $A= A_{i_0}$, it is also clear that
\[
\left\|{{f|_A}\over{\|f|_A\|}}-f\right\|<\eps.
\]
Since $u_{i_0}$ is continuous at $s_0$, there exist an open neighborhood $U$ of $s_0$  and a continuous function $h: K\longrightarrow [0,1]$ such that $u_{i_0}(s)>1-{{\eps}}$ for all $s\in U$, $h(s_0)=1$ and $h(U^c)=0$. We define a weak$^*$-continuous map $M:K\longrightarrow L_\infty(\mu)$ by
$$
M(s)(\omega)=T^*\delta_s(\omega)+\chi_A(\omega)h(s)(1-T^*\delta_s(\omega)) \qquad \bigl(\omega\in \Omega,\ s\in K\bigr).
$$
We note that $M(s_{0})=1$ for all $\omega \in A$. It is also easy to  get that $$
\|M(s)(\omega)-T^*\delta_s(\omega)\|= \|\chi_A(\omega)h(s)(1-T^*\delta_s(\omega))\|< \eps \quad \text{and} \quad \sup_{s\in K}\|M(s)\|= 1.
$$
Let $S$ be the operator represented by the function $M$. Then $S$ satisfies  condition (2), $S\left({{f|_A}\over{\|f|_A\|}}\right)(s_{0})=1$ and $\|S-T\|<\eps$.
\end{proof}

As shown in \cite{Iwa2}, the Dunford-Pettis-Phillips Theorem implies that every weakly compact operator $T$ in from $L_1(\mu)$ to an arbitrary Banach space $Y$ has separable range, hence the range of its weakly compact adjoint $T^*$ is also separable and so $T$ satisfies condition (2). On the other hand, there are Bochner representable operators which do not satisfy the condition (2) (see \cite{Iwa2}). Indeed, let $\mu$ be a strictly positive probability measure on $\mathbb{N}$ and consider the operator $T\in \mathcal{L}(L_1(\mu), C(\{0,1\} ^{\mathbb{N}})$ defined by $Tf(s) = \int f(n)\pi_n(s)\,  d\mu(n)$, where $\pi_n$ be the $n$-th natural projection on $\{0,1\}^\mathbb{N}$. Then $T$ is Bochner representable, while  $\{ T^*\delta_s : s\in G\}$ is non-separable in $L_\infty(\mu)$ for every uncountable subset $G$ of $\{0,1\}^\mathbb{N}$.

Finally, let us comment that it is also observed in \cite{Iwa2} that  if $K$ has a countable dense subset of isolated points, then condition (2) is automatically satisfied for all $T\in \mathcal{L}(L_1(\mu), C(K))$. Actually, in this case, $C(K)$ has the so-called property $(\beta)$ and then the pair $(X, C(K))$ has the BPBp for all Banach spaces $X$ \cite[Theorem~2.2]{AAGM2}.

It would be of interest to characterize those topological Hausdorff compact spaces $K$ such that $(X,C(K))$ has the BPBp for every Banach space $X$.

\end{document}